\documentclass[a4paper,11pt]{article}
\UseRawInputEncoding

\usepackage{amsmath}
\usepackage{amssymb}
\usepackage{amsthm}
\usepackage{cite}
\usepackage{graphicx}
\usepackage[footnotesize,centerlast]{subfigure}
\usepackage[small,centerlast]{caption2}
\usepackage{mathtools}
\usepackage[colorlinks=true,linkcolor=blue,citecolor=blue]{hyperref}

\usepackage[all,pdf]{xy}
\usepackage[dvipsnames]{xcolor}

\makeindex
\newtheorem{thm}{\textbf{Theorem}}[section]
\newtheorem{coro}[thm]{\textbf{Corollary}}
\newtheorem{lem}[thm]{\textbf{Lemma}}
\newtheorem{prop}[thm]{\textbf{Proposition}}
\newtheorem{defn}[thm]{\textbf{Definition}}

\newtheorem{remark}[thm]{\textbf{Remark}}

\newcommand{\Def}{Definition}
\newcommand{\Thm}{Theorem}
\newcommand{\Prop}{Proposition}

\newcommand{\Coro}{Corollary}

\newcommand{\iof}{if and only if }
\newcommand{\wrt}{with respect to }

\newcommand{\ptl}{\partial}

\newcommand{\Teich}{Teichm\"{u}ller }
\newcommand{\FN}{Fenchel-Nielsen coordinates}
\newcommand{\Sh}{shearing coordinates}
\newcommand{\WP}{Weil-Petersson }
\newcommand{\homeo}{homeomorphism}

\newcommand{\rep}{representation}
\newcommand{\sccv}{simple closed curve}
\newcommand{\mf}{measured foliation}
\newcommand{\cptspp}{compactly supported}
\numberwithin{equation}{section}

\newcommand{\Tgn}{\mathcal{T}_{g,n}}

\title{Counting mapping class group orbits\\ under shearing coordinates}
\author{Sicheng Lu and Weixu Su} 
\date{\today }

\begin{document}
\maketitle

\begin{abstract}
  Let $S_{g,n}$ be an oriented surface of genus $g$ with $n$ punctures, where $2g-2+n>0$
  and $n>0$.
  Any ideal triangulation of $S_{g,n}$ induces a global parametrization of the \Teich space $\mathcal{T}_{g,n}$
called  the shearing coordinates.
  We study the asymptotics of the number of the mapping class group orbits \wrt the standard Euclidean norm of the \Sh.
  The result is based on the works of Mirzakhani. 
\end{abstract}
\tableofcontents

\section{Introduction}
\indent

There are plenty of results about counting closed geodesics on hyperbolic surfaces.
One of the most significant results is Mirzakhani's count of simple closed geodesics.
Let $X$ be a complete hyperbolic metric on $S_{g,n}$ and let $\gamma$ be a simple closed curve.
Denote the hyperbolic length of the geodesic \rep\ of $\gamma$ on $X$ by $\ell_\gamma(X)$.
Let
\[ s_X(L,\gamma):=  \#
\{\ \beta\in\mathrm{Mod}_{g,n}\cdot\gamma\ |\ \ell_\beta(X)\leqslant L\ \} \]
 be the number of simple closed curves in the mapping class group orbit of $\gamma$ with hyperbolic length at most $L$.
Mirzakhani \cite{Mir08}  proved that, as $L\rightarrow+\infty$,
\[ s_X(L,\gamma) \sim \frac{n_\gamma\cdot B(X)}{b_{g,n}} L^{6g-6+2n}\ . \]
In the above formula, the coefficient $n_\gamma$ is determined by the topological type of the curve,
and $B(X)$ is an integrable function on the moduli space, endowed with the Weil-Petersson volume form.
The integration of $B(X)$ defines the constant $b_{g,n}$.
The above result is extended to arbitrary closed curves or multi-curves by Mirzakhani \cite{Mir16}, see also \cite{Ara21}.
There is also a different proof from other viewpoint, see \cite{ES19}.

There are two well-known parametrizations of the Teichm\"uller space related to the Weil-Petersson symplectic form.
One is the Fenchel-Nielsen coordinates, which is defined by choosing a
pants decomposition. The other is the shearing coordinates associated to an
ideal triangulation (or, in general, a maximal geodesic lamination) \cite{Thu86}.
As shown by Mirzakhani \cite{Mir16}, by counting the mapping class group orbit of a fixed hyperbolic surface in Teichm\"uller space,
one can understand the distribution of lengths and twists of curves in a random pants decomposition. A similar question is how a random ideal triangulation of a hyperbolic surface looks like.

In this paper, we  count the number of the mapping class group orbits in the shearing coordinates.
The main result is the following:


\begin{thm}\label{thm:main}
Let $\Delta$ be a given ideal triangulation of  $S_{g,n}$, where $2g-2+n>0$ and $n>0$.
Let $Sh_\Delta: \mathcal{T}_{g,n}\rightarrow\mathbb{R}^{6g-6+3n}$
be the associated shearing coordinates of the Teichm\"uller space.
Then for any $X\in\mathcal{T}_{g,n}$, as $L\rightarrow+\infty$, we have:
\begin{equation}\label{eq:counting}
\#\left\{\  \phi\in \mathrm{Mod}_{g,n}\
\big| \  \|Sh_\Delta(\phi \cdot X)\|\leqslant L\ \right\}
\ \sim\ \frac{n_\Delta \cdot B(X)}{b_{g,n}}
{L^{6g-6+2n}}\ .
\end{equation}
Moreover, the coefficient $n_\Delta$ is determined by the topological type of $\Delta$ and
can be expressed as
\[
n_\Delta= \mu_{wp}\{\ Y\in\Tgn\ |\ \|Sh_\Delta(Y)\|\leqslant 1\ \}.
\]
\end{thm}
Here $\|\cdot\|$ is the standard Euclidean norm of $\mathbb{R}^{6g-6+3n}$, and $\mu_{wp}$ denotes the Weil-Petersson volume form on the
Teichm\"uller space.
Note that for any mapping class $\phi\in \mathrm{Mod}_{g,n}$,
$Sh_\Delta(\phi\cdot X)=Sh_{\phi^{-1}\cdot \Delta}( X)$. Thus \eqref{eq:counting}
counts  the number of mapping class group orbit of a given ideal triangulations.

\bigskip

To prove Theorem \ref{thm:main}, we use the following result of Mirzakhani \cite[Theorem 1.1]{Mir16},
which is later generalized by Arana-Herrera \cite[Theorem 5.5]{Ara20}.

\begin{thm}{\rm \cite{Mir16,Ara20}} \label{thm:count}
Let $\mathcal{F}:\mathcal{T}_{g,n}\rightarrow\mathbb{R}_+$ be a positive, continuous, proper function that is asymptotically piecewise linear and bounding with respect to the Fenchel-Nielsen coordinates. Then
\begin{equation}
\lim_{L\to + \infty}
\frac{ \#\left\{\ \phi\in \mathrm{Mod}_{g,n}\ \big|\ \mathcal{F}(\phi\cdot X)\leqslant L\ \right\} }
{L^{6g-6+2n}}
=
\frac{B(X)\cdot r(\mathcal{F})}{b_{g,n}},
\end{equation}
where
\begin{equation}\label{eq:coeff}
r(\mathcal{F}):=
\lim_{L\to + \infty}
\frac{ \mu_{wp} \left\{ Y\in\mathcal{T}_{g,n}\ \big|\ \mathcal{F}(Y)\leqslant L \right\} }
{L^{6g-6+2n}}\ .
\end{equation}
\end{thm}
See \S \ref{sec:apl} and \S \ref{sec:bound} for the definitions of \emph{asymptotically piecewise linear} and \emph{bounding}  functions, respectively. The most important examples are hyperbolic lengths of closed curves.
To apply Theorem \ref{thm:count}, we show that the \Sh\ satisfy the following properties:

\noindent

\emph{(C1) The shear on each edge of $\Delta$ is asymptotically piecewise linear \wrt the \FN.}
We observe that each shear can be described by an asymptotically piecewise linear function of the hyperbolic lengths of some closed curves.
The proof is presented in \S \ref{sec:shearing}.

\medskip

\emph{(C2). The Euclidean norm of the
\Sh\ is bounding \wrt the \FN.} The proof is presented in \S \ref{sec:bounding}.
We first give an equivalent definition of the bounding condition. Then we express the length functions in terms of the \Sh, again in an asymptotically linear way.

\bigskip

To compute the coefficient $r(\mathcal{F})$ in Theorem \ref{thm:count} when $\mathcal{F}$ is the shearing norm, we show that the Weil-Petersson volume form is equal to
Euclidean volume form under the \Sh,
up to a scaling constant.
This is done in \S \ref{sec:volume}.

\bigskip

\noindent
{\bf Acknowledgements. } We would like to thank for useful discussions, and for numerous useful comments and corrections.

\section{Preliminaries}

\subsection{\Teich space and  \FN}
\indent

We recall some basic notions from the theory of \Teich spaces.
For more details, see \cite{FM,Hu}.

Given a topological surface $S_{g,n}$, its \emph{\Teich space} $\mathcal{T}_{g,n}$ is the space of all complete hyperbolic metrics up to isotopy.
More precisely, a point in $\mathcal{T}_{g,n}$ is an equivalence class of pairs $(f,\Sigma)$, where $f$ is an orientation-preserving \homeo\ from $S_{g,n}$ to a complete hyperbolic surface $\Sigma$.
Two pairs $(f_1,\Sigma_1)$ and $(f_2,\Sigma_2)$ are equivalent \iof $f_2\circ f_1^{-1}$ is homotopic to an isometry from $\Sigma_1$ to $\Sigma_2$.

Let $\gamma$ be a closed curve on $S_{g,n}$, which is neither homotopic to a point nor to a puncture.
Given $X\in\Tgn$ represented by a pair $(f,\Sigma)$, the curve $f(\gamma)$ is  freely homotopic to a unique closed geodesic on $\Sigma$.
The hyperbolic length of $f(\gamma)$ on $\Sigma$
depends on the equivalence class of the pair. Thus it defines a function on $\Tgn$ called the \emph{length function} of $\gamma$, denoted by $\ell_\gamma$.

The hyperbolic length can also be expressed by the trace of matrix.
If $X\in\Tgn$ is corresponding to a Fuchsian \rep\
$\rho_X:\pi_1(S_{g,n})\to \mathrm{PSL}(2,\mathbb{R})$,
then for any closed curve $\gamma\in\pi_1(S_{g,n})$, we have
\begin{equation}\label{eq:length-tr}
\big| \mathrm{tr}\big( \rho_X(\gamma) \big)\big|
= \cosh\left( \frac{\ell_\gamma(X)}2 \right)
\ .
\end{equation}

\medskip
Let $\mathrm{Mod}_{g,n}$ be the \emph{mapping class group} of $S_{g,n}$, i.e. the group of isotopy classes of orientation-preserving self-\homeo s leaving each puncture fixed. Then
$\mathrm{Mod}_{g,n}$ acts on $\Tgn$ by changing the markings.
If $X\in\Tgn$ is represented by $(f,\Sigma)$ and $\phi\in\mathrm{Mod}_{g,n}$,
then $\phi \cdot X \in \Tgn$ is represented by $(f\circ \phi^{-1},\Sigma)$.
In particular, for a closed curve $\gamma$, we have $\ell_\gamma(\phi \cdot X)=\ell_{\phi^{-1}\cdot \gamma}(X)$.
The action of $\mathrm{Mod}_{g,n}$ on $\Tgn$ is properly discontinuous, thus the orbit of any point is discrete.

\medskip

Recall that a  pair of pants is a topological surface 
homeomorphic to $S_{0,3}$.
A \emph{pants decomposition} of $S_{g,n}$ is a set of disjoint \sccv s which decompose the surface into pairs of pants.
Let $\mathcal{P}=\{\alpha_i\}_{i=1}^{3g-3+n}$ be  a pants decomposition of $S_{g,n}$.
The \emph{\FN} adapted to $\mathcal{P}$ consist of the length functions of $\alpha_i$ and the twist parameters $\tau_{\alpha_i}$ along $\alpha_i$. 
For the precise definition of twist parameter, see \cite[$\mathbf{\S}$3.3]{Bu}.
We have a homeomorphism
\begin{align*}
\mathrm{FN}_\mathcal{P}:
\Tgn&\longrightarrow\mathbb{R}_+^{3g-3+n}
\times\mathbb{R}^{3g-3+n}\\
X&\longmapsto \big( l_{\alpha_i}(X), \tau_{\alpha_i}(X) \big)\ .
\end{align*}

The Fenchel-Nielsen coordinates induce a canonical symplectic 2-form on $\Tgn$, which is called the \WP symplectic form. A remarkable fact due to Wolpert \cite{Wol82} is that
the \WP symplectic form does not depend on the choice of the pants decomposition. Thus it gives a volume form $\mu_{wp}$
on $\Tgn$ which is invariant under the action of the mapping class group.

\subsection{Asymptotically piecewise linear functions}\label{sec:apl}
\indent

We introduce the notion of asymptotically piecewise linear function, following \cite[\S 4]{Mir16}.

A \emph{closed cone} $\mathcal{C}\subset\mathbb{R}^k$ is a noncompact closed region bounded by finitely many hyperplanes, that is,
\[\mathcal{C}=\bigcap_{i=1}^{m}\big\{
\ \vec{x}\in\mathbb{R}^k \ \big|\
\mathcal{R}_i(\vec{x})\geqslant0\ \big\} ,\]
where each $\mathcal{R}_i(\vec{x})=r_i^1 x_1+\cdots+r_i^k x_k$ is a linear function.

We say that $\vec{x}$ \emph{tends to infinity in $\mathcal{C}$}, denoted by $\vec{x}\to \mathcal{C}_\infty$,
if $\vec{x}$ stays within the closed cone $\mathcal{C}$ and
\[\min_{i=1\cdots m}\big\{\ \mathcal{R}_i(\vec{x})\ \big\}\rightarrow +\infty\ .\]
Geometrically, this means that $\vec{x}$ stays asymptotically away from the boundary of $\mathcal{C}$.

Let $F:\mathcal{C}\to\mathbb{R}$ be a function on a closed cone.
We say that $F$ is \emph{asymptotically linear}
if there exists a linear function $\mathcal{L}:\mathcal{C}\to\mathbb{R}$ and  some real number $c$ such that
\[\lim_{\vec{x}\to \mathcal{C}_\infty}
\big( F(\vec{x}) - \mathcal{L}(\vec{x}) \big) = c\ .\]
For this we write $F\sim\mathcal{L}$ in $\mathcal{C}$.
Note that $F\sim\mathcal{L}$ in $\mathcal{C}$ \iof for any $\varepsilon>0$, there exists $A>0$ such that
\[ \big|F(\vec{x}) - \big(\mathcal{L}(\vec{x})+c\big)\big| <\varepsilon \]
for any $\vec{x}\in\mathcal{C}$ satisfying $\min\{\mathcal{R}_i(\vec{x})\}>A$.

Roughly speaking, being asymptotically linear means that, far away from the hyperplanes, the function behaves asymptotically like a linear function.
Simple examples of asymptotically linear functions are
$\cosh^{-1}(e^x)$ and $\sinh^{-1}(e^x)$.
We have
\[\cosh^{-1}(e^x) \sim x,\ \sinh^{-1}(e^x) \sim x\ \mathrm{in}\ \mathbb{R}_+\ .\]

A function is defined to be \emph{asymptotically piecewise linear}, if one can divide its domain of definition into finitely many closed cones such that the function is asymptotically linear on each cone.
Similarly, a vector-valued function $\mathbf{F}$ is asymptotically (piecewise) linear if each component is asymptotically (piecewise) linear.
Equivalently, there exists a linear transformation $\mathcal{L}:\mathbb{R}^k\to\mathbb{R}^l$ and $\vec{c}\in\mathbb{R}^l$ such that
$$\| \mathbf{F}(\vec{x})- \big( \mathcal{L}(\vec{x})+\vec{c}\ \big) \|
\to0$$ as $\vec{x}\to \mathcal{C}_\infty$.
We will use the abbreviation ``{A(P)L}" for
"asymptotically (piecewise) linear" throughout this paper.

The following composition law is easy to prove.
\begin{prop}[{\bf Composition law for APL functions}]\label{prop:apl-comp}
Let $\mathcal{C}$ be a closed cone in $\mathbb{R}^k$ and let $F_i:\mathcal{C}\to\mathbb{R}$ be APL functions for $i=1,\cdots,m$.
Let $\mathcal{C}'$ be a closed cone in $\mathbb{R}^m$ and let $G:\mathcal{C'}\to \mathbb{R}$ be an APL function.
Assume that for all $\vec{x}\in\mathcal{C}$,
$\big( F_1(\vec{x}) ,\cdots, F_m(\vec{x} ) \big) \in \mathcal{C'}$.
Then $H:= G\big( F_1 ,\cdots, F_m \big)$
is again an ALP function.
\end{prop}

Given a family $\mathcal{M}$ of APL funcitons,
let $\mathfrak{F}$ be the set of functions generated by  $\mathcal{M}$, under arithmetic operations, rational multiplication or $N$-th root:
\[f\pm g\ ,\ f\cdot g\ ,\ \frac f g\ ,\ r\cdot f\ ,\ \sqrt[N]{f}\ . \]
Then for each $f\in\mathfrak{F}$, the functions
$\sinh^{-1}(e^f),\cosh^{-1}(e^f)$ are APL.
In this paper, all APL functions we considered are obtained from such construction.

The following basic result is due to Mirzakhani \cite[Theorem 4.1]{Mir16}:

\begin{prop}\label{prop:length-apl}
For any closed curve $\gamma$ on $S_{g,n}$, the hyperbolic length function $\ell_\gamma:\mathcal{T}_{g,n}\to\mathbb{R}_+$ is an APL function \wrt any given \FN.
\end{prop}
In fact, Mirzakhani \cite{Mir16} shows that the transformation between \FN\ associated to any two different pants decomposition is APL.
Thus the property of being APL \wrt \FN\ does not depend on the the choice of the pants decomposition.

\section{The Shearing Coordinates}\label{sec:shearing}
\indent

The \Sh\ of $\Tgn$ was introduced by Thurston\cite{Thu86}.
We will show that, for any ideal triangulation, the shear on each edge is an APL function \wrt the \FN.

\subsection{Shear between adjacent triangles}
\indent

Let $T_1,T_2$ be a pair of adjacent ideal geodesic triangles in the hyperbolic plane, with a common edge $c$.
Note that all ideal geodesic triangles are isometric.
An ideal triangle has a unique inscribed circle tangent to its three edges.
Denote the tangent points of the inscribed circles on the common edge $c$ by $p_i\ (i=1,2)$.

\begin{defn}\label{def:shear}
The shear on the common edge $c$ of two adjacent ideal triangles $T_1,T_2$ is the signed distance from $p_1$ to $p_2$, denoted by
$s_c$.
Here $c$ is oriented such that $T_1$ is on the left and $T_2$ is on the right.
\end{defn}

\begin{figure} [h]
  \centering
  \includegraphics[scale=1.4]{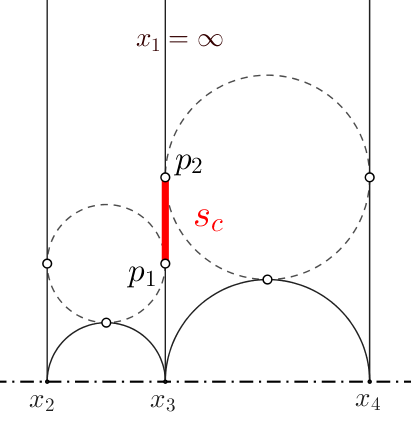}
  \caption{Shear between two adjacent ideal triangles. Here $s_c>0$ . }\label{fig:shear}
\end{figure}

One can check that interchanging the order of $T_1,T_2$ does not change the shear.
See Figure \textbf{\ref{fig:shear}} for the case of a positive shear.

The following formula relates shear with cross-ratio.
We adopt the upper half plane model for the hyperbolic plane. The ideal boundary $\ptl\mathbb{H}$ is identified with $\mathbb{R}\cup \{\infty\}$.
Denote the geodesic with two different end points $x,y\in\ptl\mathbb{H}$  by $[x,y]$.
And denote the ideal triangle with three different ideal vertices $x,y,z\in\ptl\mathbb{H}$  by $[x,y,z]$.

\begin{prop}\label{prop:sh-in-cr}
Let $x_1,x_2,x_3,x_4$ be four distinct points on $\ptl\mathbb{H}$, in counterclockwise order.
Let
$T_1=[x_1,x_2,x_3]$ and
$T_2=[x_1,x_3,x_4]$, with the common edge $c=[x_1,x_3]$. Then
\begin{equation}\label{eq:sh-in-cr}
s_c=\ln
\frac{(x_1-x_2)(x_3-x_4)}{(x_1-x_4)(x_2-x_3)}\ .
\end{equation}
\end{prop}

\subsection{The shearing coordinates}
\indent

Let $\Delta$ be an \emph{ideal triangulation} on $S_{g,n}$.
For any $X\in\Tgn$ represented by $(f,\Sigma)$, $f(\Delta)$ is homotopic to an ideal geodesic triangulation of $\Sigma$.
For each edge $c\in\Delta$, the shearing on $f(c)$ under the hyperbolic metric of $\Sigma$ is denoted by $s_c(X)$.
(To define $s_c(X)$, we can lift the map $f: S_{g,n}\to \Sigma$ to the universal covers.)
Note that $s_c(X)$ is independent of the choice of the representation $(f,\Sigma)$, thus well-defined on $\Tgn$.

We can recover the hyperbolic structure by gluing those hyperbolic ideal triangles with the data of triangulation and shears.
See Figure \ref{fig:ExpS11} for an example, which shows an ideal triangulation of $S_{1,1}$ in the universal covering space.

\begin{figure} [tb]
  \centering
  \includegraphics[scale=0.88]{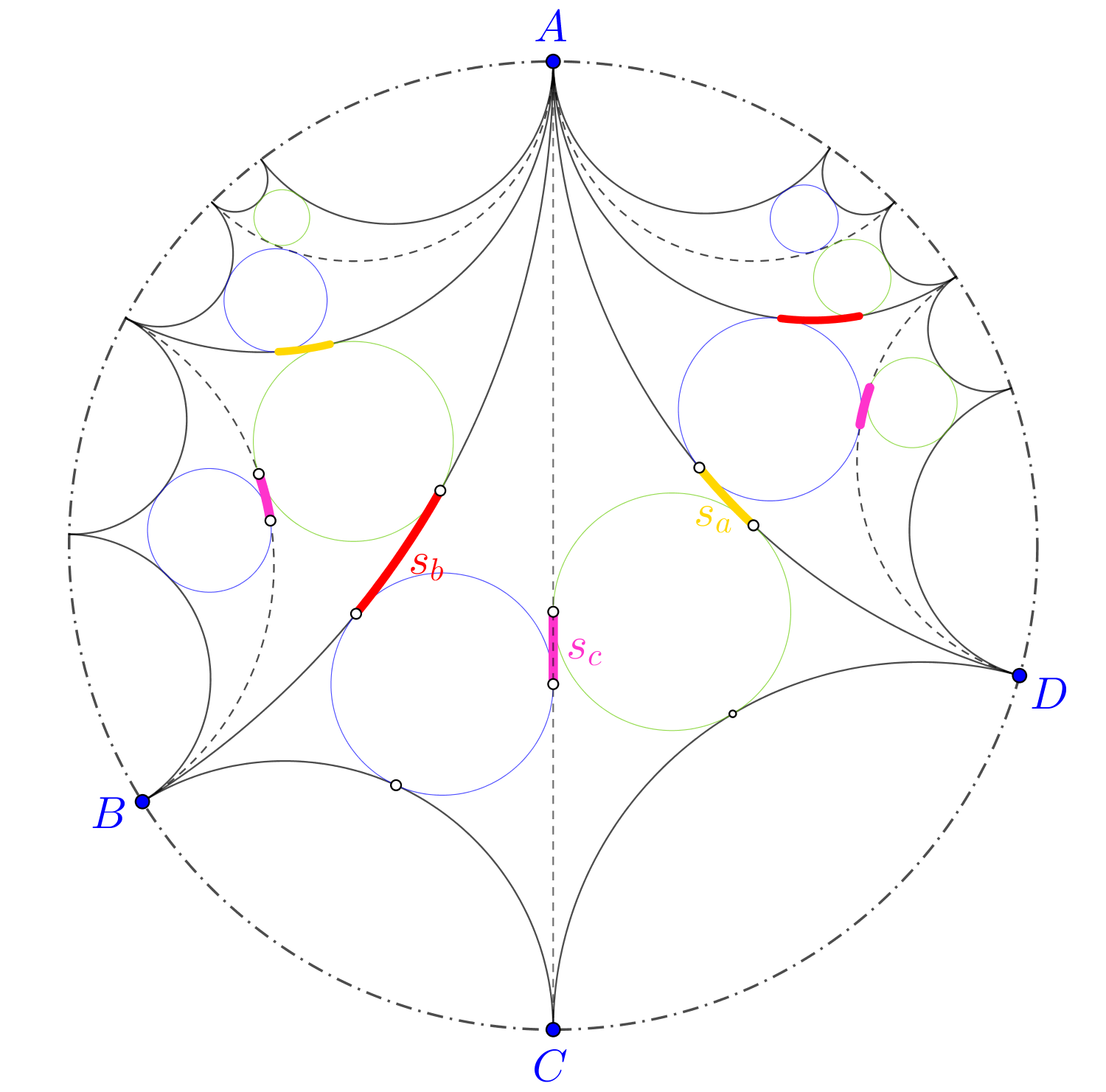}
  \caption{An ideal triangulation of $S_{1,1}$ in $\mathbb{D}$. The ideal quadrilateral $ABCD$ is a fundamental domain. The shear on each edge is labelled as a colored segment.}\label{fig:ExpS11}
\end{figure}

\begin{thm}
{\rm\cite[Theorem 3.6]{BBFS13}}
The map $Sh_\Delta$
$:\mathcal{T}_{g,n}\to
\mathbb{R}^{6g-6+3n}$, defined by \[Sh_\Delta(X):=(s_c(X))_{c\in\Delta}\ ,\]
is a \homeo\ onto its image.
\end{thm}

The shear parameters are not all independent.
At each puncture, the completeness of the hyperbolic structure induces a linear equation of the shears on the edges emitting from it. In fact£¬ the sum of these shears should be zero. Since there are $n$ punctures,
there are $(6g-6+2n)$ independent parameters,
 which coincides with the dimension of $\mathcal{T}_{g,n}$.

\begin{prop}
The \Sh\ reduces to a \homeo\
from $\Tgn$ to a   linear subspace
$\mathcal{C}_\Delta$ of dimension $(6g-6+2n)$.
\end{prop}
See \cite{BBFS13} for the proofs.

\subsection{Relation between shear and length}
\label{sec:sh-to-l}

\indent

Let $X=(f,\Sigma)$ be a point in $\Tgn$,
 and let $\Gamma$ be a Fuchsian group such that $\Sigma=\mathbb{H}/\Gamma$.
In the following, we use the trace formula (\ref{eq:length-tr}) to represent the shear as a function of hyperbolic lengths.
$\Gamma$.

Since each ideal vertex is corresponding to a fixed point of some parabolic element in $\Gamma$,
we first recall some basic properties of parabolic elements in $\mathrm{SL}_2(\mathbb{R})$.

\begin{lem}\label{lem:parab}
Suppose that $g=\left[\begin{smallmatrix}
   A & B  \\
   C & D  \\
\end{smallmatrix}\right]
\in \mathrm{SL}_2(\mathbb{R})$ is a parabolic element. Then\\
1. $\left| \mathrm{tr}\left( g \right) \right|=\left| A+D \right|=2$; \\
2. The unique fixed point of $g$ on the ideal boundary is $x=\frac{A-D}{2C}$.
\end{lem}

\begin{prop}\label{prop:tr-fixpt}
Suppose that
\[{g}_{i}=\left[ \begin{matrix}
   {{A}_{i}} & {{B}_{i}}  \\
   {{C}_{i}} & {{D}_{i}}  \\
\end{matrix} \right]
\in \mathrm{SL}_2(\mathbb{R}),\ i=1,2
\]
represent two parabolic elements in the Fuchsian group $\Gamma$, with fixed point $x_1,x_2\neq\infty$ on the boundary. Then
\begin{equation}
{\rm tr}\left( {{g}_{1}}\cdot{{g}_{2}} \right)=\frac{1}{2}{\rm tr}\left( {{g}_{1}} \right){\rm tr}\left( {{g}_{2}} \right)-{{C}_{1}}{{C}_{2}}{{\left( {{x}_{1}}-{{x}_{2}} \right)}^{2}}.
\end{equation}
In particular, ${{C}_{1}}{{C}_{2}}{{\left( {{x}_{1}}-{{x}_{2}} \right)}^{2}}$ is invariant under conjugation.
\end{prop}

\begin{proof}
By Lemma \ref{lem:parab}, one gets the equations for $i=1,2$:
\begin{align*}
   & A_iD_i-B_iC_i=1,\\
   & {{\left( A_i+D_i \right)}^{2}}=4,\\
   & A_i-D_i=2x_iC_i.
\end{align*}
Note that $x_i\neq\infty$ implies $C_i\neq0$. Solve for $B_i,D_i$ :
\begin{align*}
   & D_i=A_i-2x_iC_i,\\
   & B_i=-{{x_i}^{2}}C_i.
\end{align*}
Then
\begin{align*}\textrm{tr}\left( {{g}_{1}}\cdot{{g}_{2}} \right)&=
A_1 A_2+B_1 C_2+B_2 C_1+D_1 D_2\\
&=2(A_1-x_1C_1)(A_2-x_2C_2)-C_1C_2(x_1-x_2)^2\\
&=2\left(A_1-\frac{A_1-D_1}2\right)
\left(A_2-\frac{A_2-D_2}2\right)-C_1C_2(x_1-x_2)^2\ .
\end{align*}
\end{proof}
\begin{remark}
 If $g_1$ and $g_2$ are distinct parabolic elements, then either $g_1g_2$ or $g_1g_2^{-1}$ is hyperbolic.
 \end{remark}

\begin{coro}\label{prop:sh-in-len}
Let $T_1,T_2$ be a pair of adjacent ideal triangles  on $S_{g,n}$, with common edge $c$.
There exist four closed curves
$\gamma_1,\gamma_2,\gamma_3,\gamma_4$ on $S_{g,n}$ such that
$s_c(X)$ is an APL function of the hyperbolic length of $\gamma_i$:
\[s_c(X)=F_c\bigg( \ell_{\gamma_1}(X),
\ell_{\gamma_2}(X),\ell_{\gamma_3}(X),
\ell_{\gamma_4}(X)\bigg)\ .\]
The closed curves and the APL function depend only on the topology of $T_1,T_2$.
\end{coro}
\begin{proof}
Denote $\Sigma$ by $\mathbb{H}/\Gamma$ as before.
Let $\widetilde{T_1},\widetilde{T_2}$ be a pair of adjacent preimage of $T_1,T_2$, with ideal vertices $x_1,x_2,x_3,x_4$ in counterclockwise order and a common edge $[x_1,x_3]$.

Each $x_i$ corresponds to a primitive parabolic element $g_i\in\Gamma$.
We may also assume that all $C_i\neq0$, otherwise we can replace $\Gamma$ by an appropriate conjugation.
For $i \neq j$, let $\gamma_{ij}$ be the closed curve corresponding to the invariant geodesic axis of the hyperbolic element $g_ig_j$ or $g_ig_j^{-1}$.
Combining formula (\ref{eq:length-tr}) with \Prop\ \ref{prop:tr-fixpt}, we have:
\[-{{C}_i}{{C}_j}\left( {x}_i-{x}_j \right)^2=
\pm2\cosh\left( \frac{\ell_{\gamma_{ij}}(X)}2 \right)-\frac{{\rm tr}(g_i){\rm tr}(g_j)}2,\]
where ${\rm tr}(g_i)=\pm 2$ are constants.

Since $x_i\neq0$, we can rewrite formula (\ref{eq:sh-in-cr}) as
\begin{align*}
s_c(X)&=
\frac12 \ln \frac
{C_1C_2(x_1-x_2)^2 C_3C_4(x_3-x_4)^2}
{C_1C_4(x_1-x_4)^2 C_2C_3(x_2-x_3)^2}\ .
\end{align*}
It is easy to verify that $s_c(X)$ is APL \wrt the four length functions (using the discussion after
Proposition  \ref{prop:apl-comp}).

Given a pair of ideal triangles, the choices of $g_i$'s and $\gamma_{ij}$'s only depend on the fundamental group of the surface.
Since $\Tgn$ is simply connected, when the hyperbolic structure on the surface changes continuously in $\Tgn$, all signs appeared in the above function remains the same.
As a result, the function is determined by topology.
\end{proof}

Given an ideal triangulation $\Delta$, for each edge $c\in\Delta$ we can choose four closed curves and a function $F_c$ as above.
By \Prop\ \ref{prop:length-apl}, each variable in \Coro\ \ref{prop:sh-in-len} is an APL function \wrt the \FN.
By the composition law in \Prop\ \ref{prop:apl-comp}, we obtain:

\begin{thm}\label{thm:shear-apl}
Given an ideal triangulation $\Delta$ on the surface $S_{g,n}$, the \Sh\  $Sh_\Delta:\Tgn\to\mathbb{R}^{6g-6+3n}$
is APL \wrt the \FN.
\end{thm}

\section{Mirzakhani's counting result and the bounding condition}\label{sec:bounding}
\indent

The original bounding condition in Theorem \ref{thm:count} is proposed by Mirzakhani \cite{Mir16}. It is used to reduce the counting problem in the \Teich space to a problem in some specific cone-shaped region. Here we adopt the definition in \cite{Ara20}.

\subsection{Restatement of the bounding condition}\label{sec:bound}
\indent

Under the \FN\ $(\ell_i,\tau_i)_{i=1}^{3g-3+n}$ adapted to some pants decomposition $\mathcal{P}$, the Teichm\"uller space $\Tgn$ admits a partition into countably many convex polytopes of the form
\[
\mathcal{C}_\mathcal{P}^\mathbf{m}:=
\bigg\{\ Y\in\Tgn\ \bigg|\
m_i\cdot \ell_i(Y) \leqslant \tau_i(Y) \leqslant (m_i+1)\cdot \ell_i(Y)\ \bigg\}
\]
with $\mathbf{m}:=(m_1,\cdots,m_{3g-3+n})
\in \mathbb{Z}^{3g-3+n}$. 

\begin{defn}\label{def:bounding}
A funciton $\mathcal{F}:\mathcal{T}_{g,n}\rightarrow\mathbb{R}_+$
is bounding \wrt the \FN\ $(\ell_i,\tau_i)_{i=1}^{3g-3+n}$,
if for every $Y\in\Tgn$ there exists a constant $C>0$ such that
for every $\mathbf{m}\in\mathbb{Z}^{3g-3+n}$ and every $Z\in\mathrm{Mod}_{g,n}\cdot Y\cap \mathcal{C}_\mathcal{P}^\mathbf{m} \cap \mathcal{F} ^{-1}([0,L])$,
\begin{equation}\label{eq:bounding}
\ell_i(Z)\leqslant C\cdot\frac L
{\max\{|m_i|,|m_i+1|\}}\ .
\end{equation}
\end{defn}

This means that, when the value of $\mathcal{F}$ grows, the length of the $i$-th pants curve grows at a linear rate, and it is also proportional to the twist component.

\begin{prop}\label{prop:fillike}
The bounding condition (\ref{eq:bounding}) is equivalent to the following  condition:\\
for every $Y\in\Tgn$,
\begin{equation}\label{eq:fillike}
\sup_{Z\in\mathrm{Mod}_{g,n}\cdot Y}
\left\{ \frac{\ell_i(Z)+|\tau_i(Z)|}
{\mathcal{F}(Z)} \right\}
<+\infty\ .
\end{equation}
\end{prop}

\begin{proof}
We write $\ell_i=\ell_i(Z)$ and $\tau_i=\tau_i(Z)$
for simplicity.

Suppose (\ref{eq:bounding}) holds.
When $m_i\geqslant0$ and  $Z\in\mathcal{C}_\mathcal{P}^\mathbf{m}$, we have $0 \leqslant m_i \leqslant \tau_i/\ell_i \leqslant m_i+1$.
Then
\begin{align*}
C\geqslant
(m_i+1)\frac{\ell_i}{\mathcal{F}} &\geqslant\frac{m_i+2}{2}\frac{\ell_i}{\mathcal{F}}\\
&=\frac{\ell_i+(m_i+1)\ell_i}{2\mathcal{F}}
\geqslant\frac{\ell_i+\tau_i}{2\mathcal{F}}\ .
\end{align*}
Thus
\[\frac{\ell_i+\tau_i}{\mathcal{F}}\leqslant2C\ .\]
When $m_i<0$, we have
$0\leqslant-m_i-1\leqslant|\tau_i|/\ell_i\leqslant-m_i$. Then
\begin{align*}
C\geqslant
(-m_i)\frac{\ell_i}{\mathcal{F}} &\geqslant\frac{-m_i+1}{2}\frac{\ell_i}{\mathcal{F}}\\
&=\frac{\ell_i+(-m_i)\ell_i}{2\mathcal{F}}
\geqslant\frac{\ell_i+|\tau_i|}{2\mathcal{F}}\ .
\end{align*}
So (\ref{eq:fillike}) holds.

Now suppose $\mathcal{F}$ satisfies (\ref{eq:fillike}), with upper bound $K>0$. If $Z\in\mathrm{Mod}_{g,n}\cdot Y$ with\\
$0\leqslant m_i\leqslant \tau_i/\ell_i\leqslant m_i+1$, then
\[K\geqslant\frac{\ell_i+\tau_i}{\mathcal{F}}
\geqslant\frac{\ell_i+m_i \ell_i}{\mathcal{F}}\ .\]
If $m_i\leqslant \tau_i/\ell_i\leqslant m_i+1 \leqslant 0$, then
\[K\geqslant\frac{\ell_i-\tau_i}{\mathcal{F}}
\geqslant\frac{\ell_i-(m_i+1)\ell_i}{\mathcal{F}}
=\frac{-m_i\ell_i}{\mathcal{F}}\ .\]
Thus (\ref{eq:bounding}) holds.

Note that each part of the proof utilizes one side of the condition $m_i\leqslant \tau_i/\ell_i\leqslant m_i+1$.

\end{proof}

\subsection{Relation between length and shear}
\indent

Our aim here is to prove that ${\mathcal{F}(X)}=\|Sh_\Delta(X)\|$ satisfies the
inequality (\ref{prop:fillike}).
In the following theorem, we have a formula of the length function in terms of shears. This is the key result in this paper.

\begin{thm}\label{prop:lth-in-sh}
Let $\gamma$ be a non-degenerated closed curve on $S_{g,n}$, and let $X\in\Tgn$ represented by $\mathbb{H}/\Gamma$.
Let $g\in\Gamma$ be a primitive hyperbolic element corresponding to $\gamma$.
Let $(s_1,\cdots,s_K)=Sh_\Delta(X)$ be the \Sh\ of $X$ associated to the triangulation $\Delta$.
Then $|\mathrm{tr}(g)|=2\cosh(\ell_\gamma(X)/2)$ is a polynomial of variables $\{e^{\pm s_j /2}\}_{j=1}^K$ with rational coefficients:
\[\cosh(\ell_\gamma/2)\in \mathbb{Q}
\big[e^{\pm s_1/2},\cdots,e^{\pm s_K/2}\big]\ .\]
The polynomial only depends on the topological type of $\Delta$ and $\gamma$.
\end{thm}

\begin{proof}We will give a precise algorithm to compute the matrix $g\in \mathrm{SL}_2(\mathbb{R})$.

Passing to the universal cover,
we fix an orientation of $\gamma$ and choose one intersection of $\gamma$ and $\Delta$ as an initial point
(here we have identify $\gamma$ with the axis of $g$).
Up to a conjugation, we may assume that the initial point of $\gamma$ is the complex number $i$ in the upper half plane, contained in the ideal triangle $T_1=[-1,1,\infty]$.

Traveling forward along $\gamma$,
we have a sequence of triangles
$(T_1,\cdots,T_{N+1}=gT_1)$.
Let $a_i$ be the common ideal edge of $T_i,T_{i+1}$, $i=1,\cdots,N$.
The sequence $\left(a_i\right)_{i=1}^N$ depends only on the type of $\gamma$ and $\Delta$, thus a topological data.
The shear on $a_i$ is denoted by $s_i$.
Our algorithm consists of three steps:

\medskip

\emph{(1) The initial data: ``Left $\&$ Right sequences''.}

\medskip

If $\gamma$ enters $T_i$ through one edge, then it must leave through one of the other two edges.
Given the orientation of $X$ and $\gamma$, one can tell $\gamma$ leave the triangle through the left edge or the right edge.
Thus we can say that $T_{i+1}$ lies on the left or on the right hand side of $T_i$.
Define $\varepsilon_i=+1$ if $T_{i+1}$ lies on the left, $\varepsilon_i=-1$ if on the right.
See Figure \ref{fig:LR} for illustration and examples.

\begin{figure*}[tb]
 \centering
 \subfigure[Turn left, or $\varepsilon=+$.]
  {\includegraphics[scale=1]{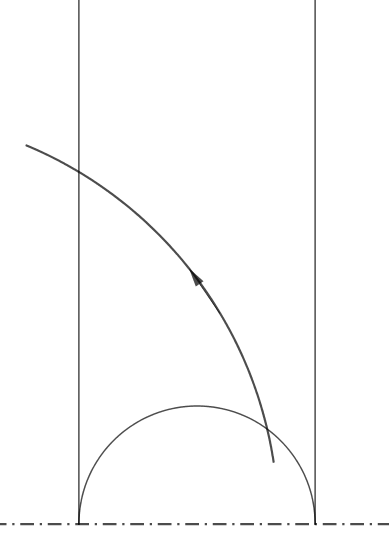}}
 \hspace{15mm}
 \subfigure[Turn right, or $\varepsilon=-$.]
  {\includegraphics[scale=1]{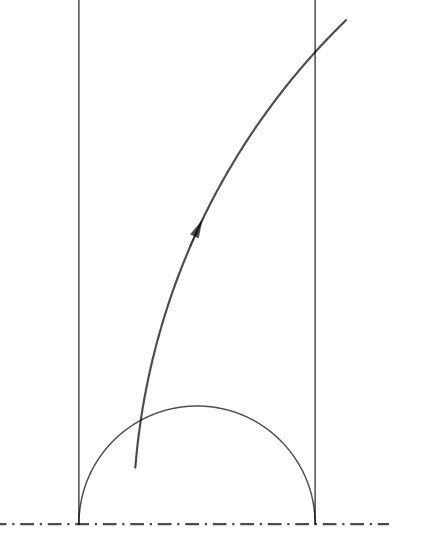}}
 \subfigure[From $T_1$ to $T_4$, the sequence is $(-,+,-).$]
  {\includegraphics[scale=1]{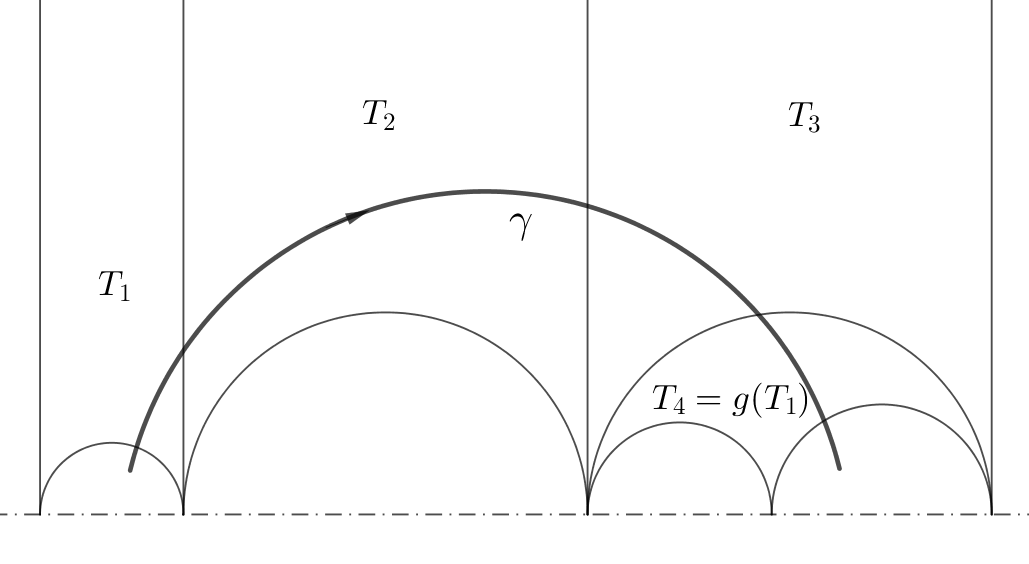}}
 \caption{The 'left-right sequences'.}
 \label{fig:LR}
\end{figure*}

Then we get a sequence of signatures
$\left(\varepsilon_i\right)_{i=1}^N$, which tells how $\gamma$ passing through each triangles along the path.

\medskip

\emph{(2) The basic matrices of shear.}

\medskip

The initial data characterizes how $\gamma$ passing through each triangle.
Now we construct the basic matrices of M\"obius transformation corresponding to the shear deformation.

\medskip

If $\gamma$ leaves through the right edge, we may first apply a M\"{o}bius transformation $R$ to map $(-1,i,1)$ into $(\infty,1+2i,1)$. Then the two triangles $T_1,R(T_1)$ have a common edge $[1,\infty]$ with coincide tangent points.
If $\gamma$ leaves through the left edge, we map $(-1,i,1)$ into $(-1,-1+2i,\infty)$. The corresponding matrices are
\[
R:=\frac12\left(
  \begin{array}{cc}
    3 & -1 \\
    1 & 1 \\
  \end{array}
\right),\
L:=\frac12\left(
  \begin{array}{cc}
    3 & 1 \\
    -1 & 1 \\
  \end{array}
\right)
\]

To unify them, we may define
\[
P_\varepsilon:=\frac12
\left(
  \begin{array}{cc}
    3 & \varepsilon \\
    -\varepsilon & 1 \\
  \end{array}
\right),\ \varepsilon=\pm1
\]
Then $R=P_{-},\ L=P_{+}$.

\medskip

We have mapped the entering edge onto the leaving edge by $P_\varepsilon$ as above.
If $\gamma$ leave through the right edge,
then shearing along the leaving edge means applying a hyperbolic transformation with fixed point $\{1,\infty\}$ and signed translation distance $s$.
In complex coordinate, the function is
$z\mapsto e^s(z-1)+1$.
Similarly, for the left edge case, the complex function is
$z\mapsto e^{-s}(z+1)-1$.
The unified matrices are
\[H_\varepsilon(s):=
\left(
  \begin{array}{cc}
    e^{\ -{\varepsilon s}/{2}} & -\sinh(s/2) \\
    0 & e^{\ {\varepsilon s}/{2}} \\
  \end{array}
\right)
\]
Some calculation shows that the third vertex of $(H_\varepsilon(s)\circ P_\varepsilon) (T_1)$ other than $\pm1,\infty$ is \\$-\varepsilon (1+2e^{-\varepsilon s})$.

In conclusion, we take ${V_\varepsilon(s)} := H_\varepsilon(s)\circ P_\varepsilon$.
It maps $T_1$ to the next triangle along $\gamma$.

\medskip

\emph{(3) The composition diagram.}

\medskip

Let $g$ be a primitive hyperbolic element, corresponding to the translation along $\gamma$ for a single period.
We now compute the matrix representation for $g$.
 It is a composition of basic matrices defined as above.

Let $\left(T_i\right)_{i=1}^{N},\left(a_i\right)_{i=1}^{N}, \left(\varepsilon_i\right)_{i=1}^{N}$ as before.
Denote by $a_0$ be the edge at which $\gamma$ enters $T_1$.
For any $i=1,\cdots,N$, there is a unique isometry $f_i$ that maps $T_i$ into $T_{i+1}$ with the edge $a_{i-1}$ matching $a_{i}$.
Then $$g=f_N \circ\cdots\circ f_1.$$
Denote $V_{\varepsilon_i}(s_{\varepsilon_i})$ by $V_i, i=1,\cdots,N$ to simplify notations.
Note that $f_1=V_1$.

We have the following diagram:
\[\xymatrixcolsep{5pc}
\xymatrix{
   T_1 \ar[r]^{f_1} &
   T_2 \ar[r]^{f_2} \ar[d]_{V_1^{-1}} &
   T_3 \ar[r]^{f_3} \ar[d]^{V_1^{-1}} &
   \cdots \ar[r]^{f_{N-1}} &
   T_N \ar[r]^{f_N} \ar[d]_{V_1^{-1}} &
   T_{N+1} \ar[d]^{V_1^{-1}} \\
   &
  T_1 \ar[r]^{V_2} &
  T_3' \ar[r] \ar[d]^{V_2^{-1}} &
  \cdots \ar[r] &
  T_N' \ar[r] \ar[d]_{V_2^{-1}}  &
  T_{N+1}' \ar[d]^{V_2^{-1}} \\
   & &
  T_1 \ar[r]_{V_3} &
  \cdots \ar[r] &
  \cdots \ar[d]_{V_{N-1}^{-1}} \ar[r] &
  \cdots \ar[d]^{V_{N-1}^{-1}} \\
   & & & &
  T_1 \ar[r]^{V_N} & T^{(N)}_{N+1}   }
\]
By going to the bottom right corner diagonally and then going up, one get
\begin{align*} g&=
V_1 \circ V_2 \circ \cdots V_{N-1} \circ V_N \circ V_{N-1}^{-1} \cdots \circ V_2^{-1} \circ V_2 \circ V_1^{-1} \circ f_1
\\
&=V_1 \circ V_2 \circ \cdots V_{N-1} \circ V_N\\
&=H_{\varepsilon_1}(s_1) P_{\varepsilon_1}\cdots
H_{\varepsilon_N}(s_N) P_{\varepsilon_N}
\ .\end{align*}

It is obvious that each element in the matrix is a polynomial of $e^{\ \pm s_i/2}$, with rational coefficients.
The polynomials are determined by topology.
\end{proof}

For the twist $\tau_i$ on each pants curve, we  have a similar result.

\begin{coro}\label{prop:tw-in-sh}
With the notations in \Def\ \ref{def:bounding} and \Thm\ \ref{prop:lth-in-sh}, we have:
\[\cosh^2(\tau_i)\in \mathbb{Q}
\big({e}^{\pm s_1/2},\cdots,e^{\pm s_K/2}\big)\ .\]
Furthermore, on each mapping class group orbit, its denominator is bounded away from 0.
\end{coro}

\begin{proof}
It is known that the twist $\tau$ along each pants curve $\gamma$
can be described in terms of the length functions of some closed curves. See \cite[Chapter 3]{Bu} and \cite{Mir16}.
In the following, to simplify notation, the length of curve is denoted by the same notation as the curve itself.

The following formulae will be used later:
\[\cosh 2x = 2\cosh^2 x -1 = 2\sinh^2 x +1,\ \sinh 2x = 2\sinh x \cosh x\ .\]

There are two types of the curves:

\medskip

(1) The curve $\gamma$ is contained in a (1,1)-type subsurface. See Figure \ref{fig:tw-in-len}(a).

\medskip

We have:
\[\begin{cases}
\ \cosh d \sinh^2(\gamma/2) = \cosh(\delta/2)+\cosh^2(\gamma/2) \\
\ \cosh (\mu/2) = \cosh (d/2) \cosh(\tau/2)\\
\end{cases}.\]
Here $\delta$ is the other boundary curve of the pants, $d$ is a segment perpendicular to $\gamma$, and $\mu$ is a simple closed curve. It follows that
\[ 1 + \cosh\tau =
\frac{(\cosh\mu +1)(\cosh \alpha-1)}
{\cosh(\delta/2)+\cosh\alpha}\ .\]

\medskip

(2) The curve $\gamma$ is contained in a (0,4)-type subsurface. See Figure \ref{fig:tw-in-len}(b).

We have:
\[\begin{cases}
\ \cosh\frac{\mu}{2} = \cosh d \sinh\frac{{\beta}}{2} \sinh\frac{{\beta}'}{2} - \cosh\frac{{\beta}}{2} \cosh\frac{{\beta}'}{2} \\
\ \cosh d = \cosh{\tau} \sinh{h} \sinh{h'} + \cosh{h}\cosh{h'} \\
\ \cosh{h}\ \sinh\frac\gamma{2} \sinh\frac{\beta}{2}\ = \cosh\frac{\delta}{2} + \cosh\frac\gamma{2} \cosh\frac{{\beta}}{2} \\
\ \cosh{h'} \sinh\frac\gamma{2} \sinh\frac{{\beta}'}{2} = \cosh\frac{{\delta}'}{2} + \cosh\frac\gamma{2} \cosh\frac{{\beta}'}{2} \\
\end{cases}. \]
Here $\beta,\delta,\beta',\delta'$ are the other boundary curves of the pants,
and $d,h,h'$ are the common perpendicular segments,  of certain topological type, from $\beta$ to $\beta'$, $\beta$ to $\gamma$ and $\beta'$ to $\gamma$, respectively.
$\mu$ is a simple closed curve.
We have
\[ \cosh\tau = \frac {\cosh d -\cosh h \cosh h'} {\sinh h \sinh h'}\ .\]

\begin{figure}[!t]
 \subfigure[]
  \centering
  \def\svgwidth{0.5\columnwidth}
\begingroup%
  \makeatletter%
  \providecommand\color[2][]{%
    \errmessage{(Inkscape) Color is used for the text in Inkscape, but the package 'color.sty' is not loaded}%
    \renewcommand\color[2][]{}%
  }%
  \providecommand\transparent[1]{%
    \errmessage{(Inkscape) Transparency is used (non-zero) for the text in Inkscape, but the package 'transparent.sty' is not loaded}%
    \renewcommand\transparent[1]{}%
  }%
  \providecommand\rotatebox[2]{#2}%
  \newcommand*\fsize{\dimexpr\f@size pt\relax}%
  \newcommand*\lineheight[1]{\fontsize{\fsize}{#1\fsize}\selectfont}%
  \ifx\svgwidth\undefined%
    \setlength{\unitlength}{209.7766092bp}%
    \ifx\svgscale\undefined%
      \relax%
    \else%
      \setlength{\unitlength}{\unitlength * \real{\svgscale}}%
    \fi%
  \else%
    \setlength{\unitlength}{\svgwidth}%
  \fi%
  \global\let\svgwidth\undefined%
  \global\let\svgscale\undefined%
  \makeatother%
  \begin{picture}(1,0.4462226)%
    \lineheight{1}%
    \setlength\tabcolsep{0pt}%
    \put(0,0){\includegraphics[width=\unitlength,page=1]{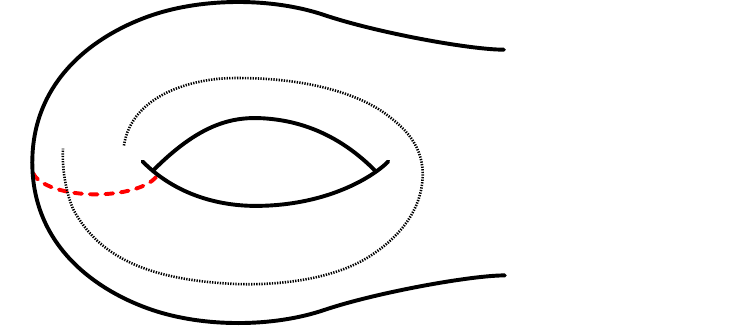}}%
    \put(-0.00161072,0.20867933){\makebox(0,0)[lt]{\lineheight{1.04999995}\smash{\begin{tabular}[t]{l}{\color{red}{$\gamma$}}\end{tabular}}}}%
    \put(0.32428142,0.39306845){\makebox(0,0)[lt]{\lineheight{1.04999995}\smash{\begin{tabular}[t]{l}{\color{blue}{$\mu$}}\end{tabular}}}}%
    \put(0.68788437,0.20216608){\makebox(0,0)[lt]{\lineheight{1.04999995}\smash{\begin{tabular}[t]{l}{\color{OliveGreen}{$\delta$}}\end{tabular}}}}%
    \put(0.54105328,0.089253){\makebox(0,0)[lt]{\lineheight{1.04999995}\smash{\begin{tabular}[t]{l}$d$\end{tabular}}}}%
    \put(0,0){\includegraphics[width=\unitlength,page=2]{S11n.pdf}}%
  \end{picture}%
\endgroup%

 \hspace{10mm}
 \subfigure[]
  \centering
  \def\svgwidth{0.45\columnwidth}
  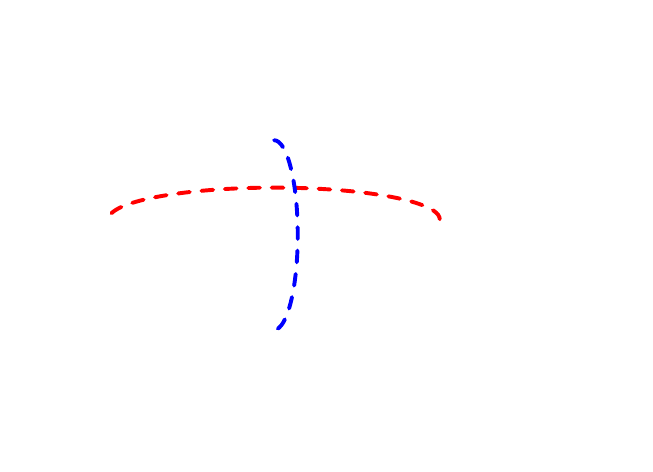
 \caption{Local closed curves describe the twist.}
 \label{fig:tw-in-len}
\end{figure}

\medskip

We have shown that $\cosh^2 \tau$ is a rational function of desired form. Note that in a certain mapping class group orbit, the length of pants curves have a lower bound.
In both of the above two cases, the denominator is a function of length of pants curves, which is bounded from below.

\end{proof}

\begin{remark}
The results and proofs in this section also hold if $\Delta$ is a maximal lamination with only isolated and closed leaves.

According to \cite{BBFS13}, for each closed leaf, there is a liner equation about its length and the shear.
And the shear on closed leaves are basically the same as twist.
These facts ensure the APL property in both directions.
\end{remark}

\subsection{The norm of shearing coordinates is bounding.}
\begin{prop}\label{prop:bound}
The norm $\|Sh_\Delta\|$ is bounding \wrt the \FN.
\end{prop}
\begin{proof}
Let $\|\vec{s} \|$ be the Euclidean norm of a vector $\vec{s}$. 
By Theorem \ref{prop:lth-in-sh}, for each $i$ there are positive rational numbers $M_i,A_i,N_i,B_i$ such that
\[\cosh(\ell_i/2)\leqslant M_i e^{A_i \| \vec{s} \|},\ \cosh^2(\tau_i)\leqslant N_i e^{B_i \| \vec{s} \|}\]
on a particular mapping class group orbit.
Thus  $\ell_i,|\tau_i|$ are bounded by linear functions of $\| \vec{s} \|$, with positive leading coefficients.
By \Prop\ \ref{prop:fillike}, $\|Sh_\Delta\|$ is bounding \wrt the \FN.
\end{proof}

\section{Weil-Petersson volume under \Sh}\label{sec:volume}
\indent

The last task is to show that under the \Sh,
the Weil-Petersson volume form is the Euclidean volume form, up to a scaling constant.
To see this, we use the cataclysms coordinates of
Thurston \cite{Thu86}.

\subsection{Weil-Petersson volume form}
\indent

A \emph{\mf}\ on a surface is a foliation with singularities together with a transverse measure, which is invariant under homotopic moving along the leaves of foliation.
Two \mf s are \emph{equivalent} if one may be
transformed to the other by isotopies moves and Whitehead moves, which allow to break down or combine the singularities.
Usually, a \mf\ refers to an equivalence class.
The space of all equivalence classes of \mf s on a topological surface is denoted by $\mathcal{MF}$.

For surfaces with punctures, we shall only consider foliations with compact support.
This means that the support of the transverse measure is bounded away from some neighbourhood of the punctures.
Let $\mathcal{MF}_0$ be the space of all equivalence classes of \cptspp\ \mf s on $S_{g,n}$.

The space $\mathcal{MF}_0$ has a piecewise linear structure.
And it admits a 2-form called the Thurston symplecitc form.
The symplectic form induces a natural volume form.
We refer to \cite{FLP} for more details on \mf s, and to \cite{PeH} for measured laminations and related topics.
\medskip

There is a close relation between the Thurston symplectic forms and the Weil-Petersson symplectic form, via the \Sh\ \cite{SoB01}.
In the special case of ideal triangulation, the relation is rather simple \cite{PaP93}. Let us describe in the following.

Given an ideal triangulation $\Delta$ and a hyperbolic surface $X\in\Tgn$, there is a foliation on $X$ whose leaves are segments of horocycles centred at the ideal vertices.
The complement of its support in each ideal triangle is a small triangle bounded by three horocycle segments of length $1$,
meeting tangentially at the tangent points of the inscribed circles.
See Figure \ref{fig:horo}.

\begin{figure*}[t]
 \centering
 \subfigure
  {\includegraphics[scale=0.55]{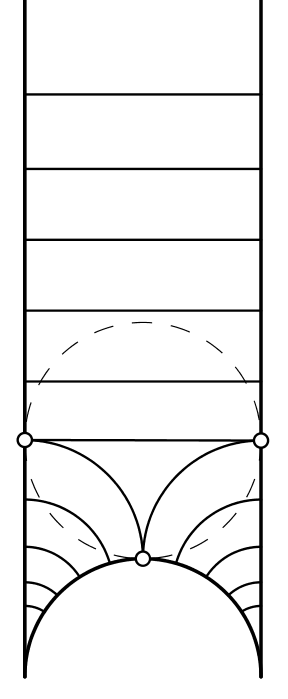}}
 \hspace{25mm}
 \subfigure
  {\includegraphics[scale=0.55]{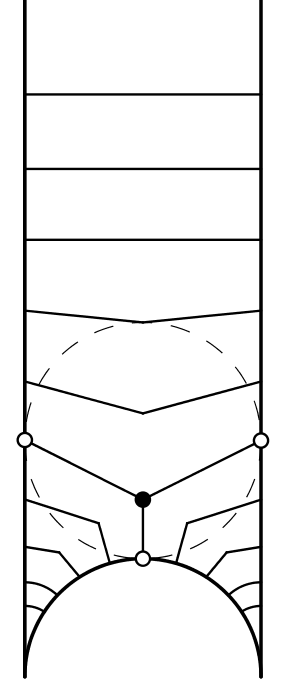}}
 \caption{Horocycle foliation in a single ideal triangle.}
 \label{fig:horo}
\end{figure*}

We can endow the horocycle foliation with a transverse measure such that the measure of any geodesic arc $I$ contained in $\Delta$ is equal to its hyperbolic length.
By collapsing each small unfoliated triangle to a 3-pronged singularity, we obtain a measured foliation in $\mathcal{MF}$.

This foliation is not \cptspp. However, the completeness of the hyperbolic metric guarantees that leaves near the cusped region must be closed.
Thus we are able to obtain a \cptspp\ foliation
by deleting all of these closed leaves paralleled to the punctures.
Denote the resulted foliation by $\mathcal{F}_\Delta(X)$. See Figure \ref{fig:HoroFoli} for an example of this process on $S_{1,2}$.

\begin{figure*}[!t]
 \centering
 \subfigure[An ideal triangulation of $S_{1,2}$ with the shear data.]
  {\includegraphics[scale=0.35]{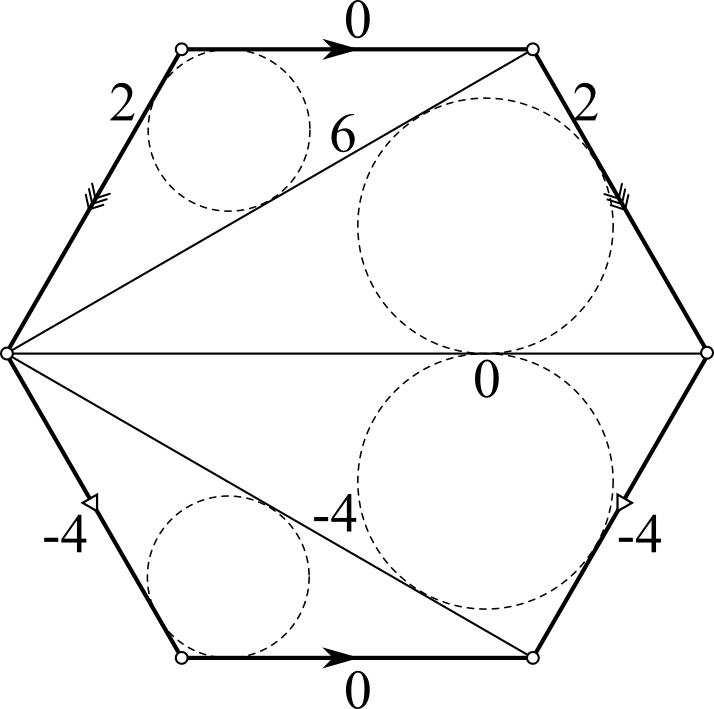}}
 \hspace{15mm}
 \subfigure[The entire horocycle foliation, with colored singular leaves.]
  {\includegraphics[scale=0.35]{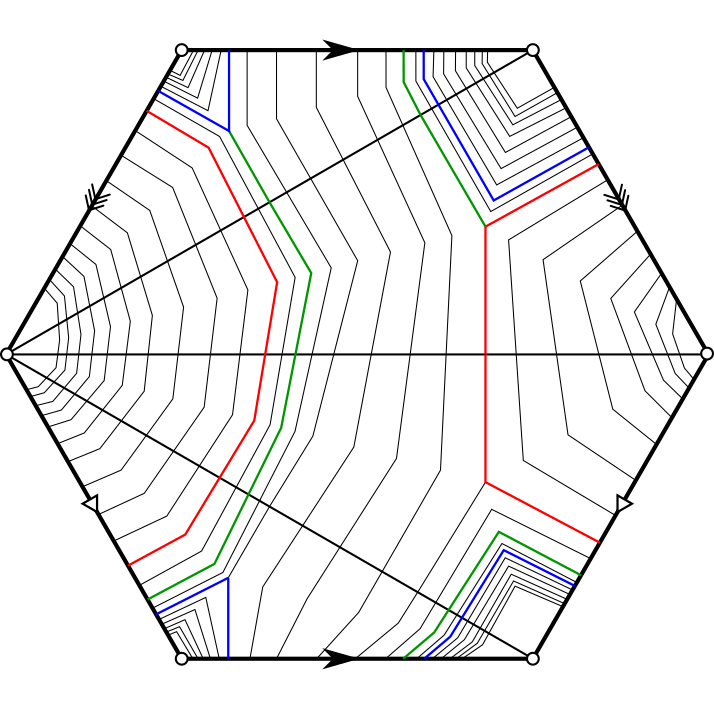}}
 \subfigure[The \mf\ $\mathcal{F}_\Delta(X)$. 
 The boundary must consist of closed singular leaves.]
  {\includegraphics[scale=0.35]{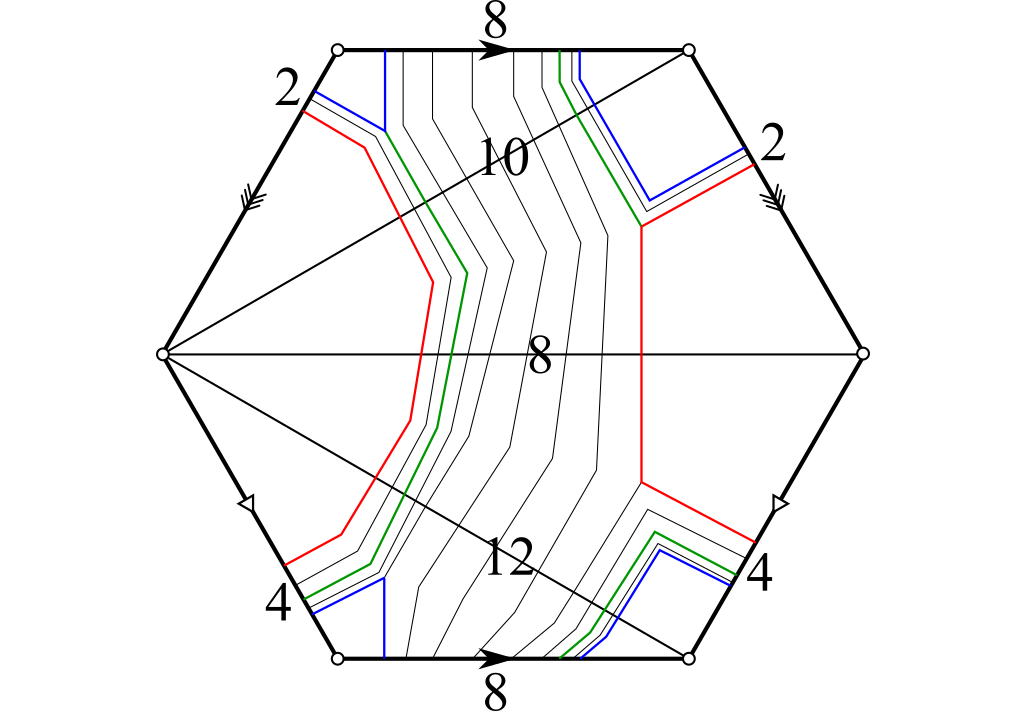}}
 \caption{From the shear data to a \cptspp\ \mf.}
 \label{fig:HoroFoli}
\end{figure*}

By \cite[\Prop\ 9.4]{Thu86} and \cite{PaP93}, the map
\begin{align*}
\mathcal{F}_\Delta: \Tgn &\longrightarrow \mathcal{MF}_0\\
X &\mapsto \mathcal{F}_\Delta(X)
\end{align*}
is a \homeo.

Denote the transverse measure of $\mathcal{F}_\Delta(X)$ on each edge $c\in\Delta$ by {$m_c(X)$}.
We can embed $\mathcal{MF}_0$  into $\mathbb{R}_{\geq0}^{6g-6+3n}$ as an Euclidean cone:
\begin{align*}
 \mathcal{MF}_0 &\longrightarrow \mathbb{R}_{\geq0}^{6g-6+3n}\\
\mathcal{F}_\Delta(X) &\mapsto \left(m_{c_i}(X)\right)_{i=1}^{N}
\end{align*}
We use the above map to define the symplectic form on $\mathcal{MF}_0$.

\begin{prop}{\rm\cite[Corollary 4.2]{PaP93}}
The \homeo\ $\mathcal{F}_\Delta$ pulls back the Thurston's symplectic form on $\mathcal{MF}_0$ to the Weil-Petersson form on $\Tgn$.
\end{prop}

Recall that $\mathcal{C}_\Delta$ is the image of the \Sh,
 as a subspace of $\mathbb{R}^{6g-6+3n}$.
By the above construction, we can consider the composition map $Sh_\Delta\circ\mathcal{F}_\Delta^{-1}$ as a map between Euclidean spaces.
The following should be equivalent to \cite[\Prop\ 9.1]{Thu86}.

\begin{prop}\label{eq:ms-to-sh}
The coordinate transformation $Sh_\Delta\circ\mathcal{F}_\Delta^{-1}$ from $\mathcal{MF}_0$ to $\mathcal{C}_\Delta$ is determined by
\[ s_e(X)=\frac12 \bigg( m_a(X)+m_c(X)-m_b(X)-m_d(X) \bigg) \]
for each edge $e\in\Delta$.
Here $a,b,e$ and $c,d,e$ are the edges of two adjacent ideal triangles, with $a,b,c,d$ in counterclockwise order.
\end{prop}

\begin{proof}
The formula follows immediately from the definition of shear and the construction of the \mf. See Figure \ref{fig:mf-in-pair}.
\end{proof}

\begin{figure}[!t]
  \centering
  \def\svgwidth{0.44\columnwidth}
\begingroup%
  \makeatletter%
  \providecommand\color[2][]{%
    \errmessage{(Inkscape) Color is used for the text in Inkscape, but the package 'color.sty' is not loaded}%
    \renewcommand\color[2][]{}%
  }%
  \providecommand\transparent[1]{%
    \errmessage{(Inkscape) Transparency is used (non-zero) for the text in Inkscape, but the package 'transparent.sty' is not loaded}%
    \renewcommand\transparent[1]{}%
  }%
  \providecommand\rotatebox[2]{#2}%
  \newcommand*\fsize{\dimexpr\f@size pt\relax}%
  \newcommand*\lineheight[1]{\fontsize{\fsize}{#1\fsize}\selectfont}%
  \ifx\svgwidth\undefined%
    \setlength{\unitlength}{651.88536253bp}%
    \ifx\svgscale\undefined%
      \relax%
    \else%
      \setlength{\unitlength}{\unitlength * \real{\svgscale}}%
    \fi%
  \else%
    \setlength{\unitlength}{\svgwidth}%
  \fi%
  \global\let\svgwidth\undefined%
  \global\let\svgscale\undefined%
  \makeatother%
  \begin{picture}(1,0.57930709)%
    \lineheight{1}%
    \setlength\tabcolsep{0pt}%
    \put(0,0){\includegraphics[width=\unitlength,page=1]{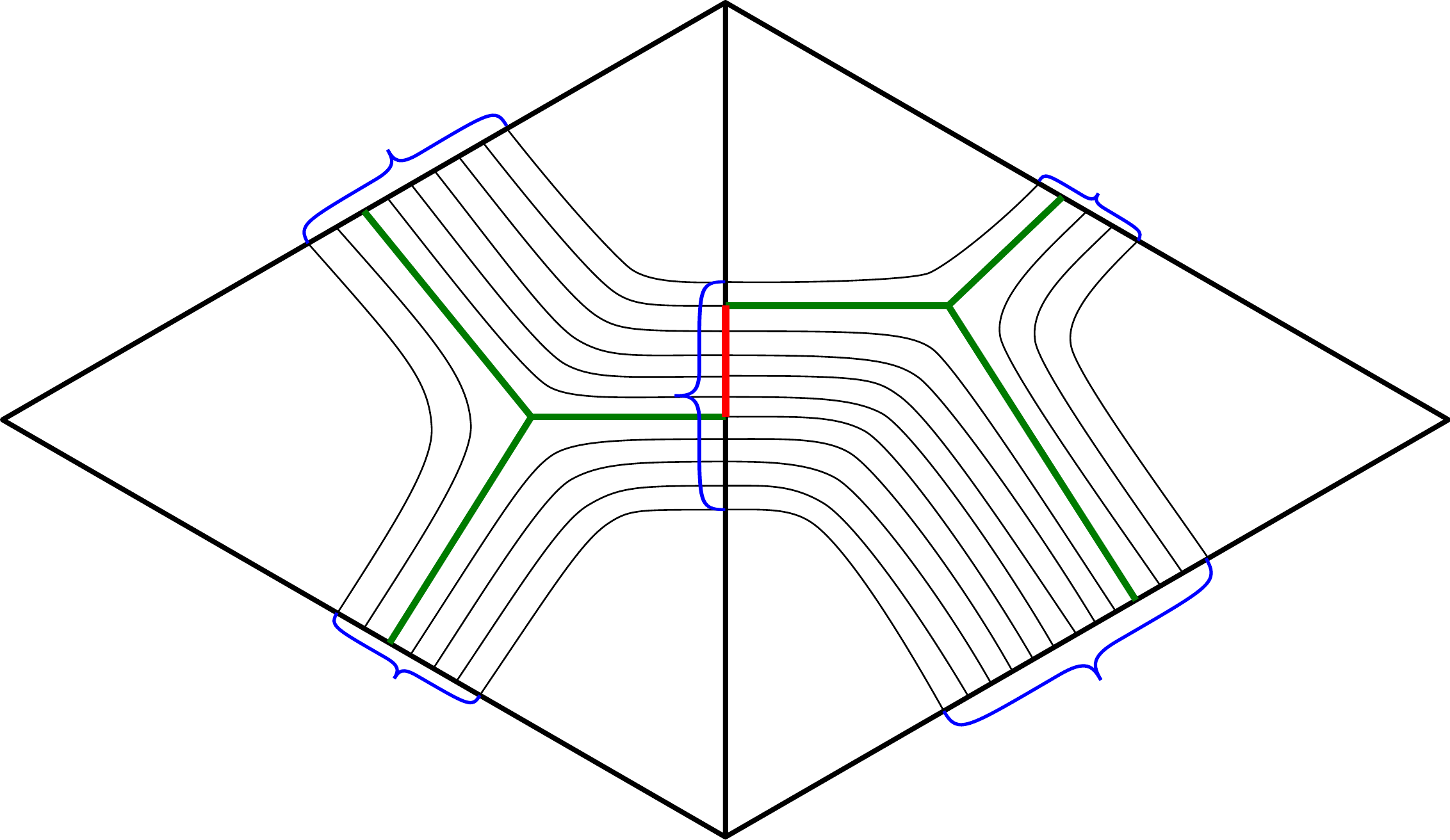}}%
    \put(0.514926,0.38095705){\color[rgb]{0,0,0}\makebox(0,0)[lt]{\lineheight{1.25}\smash{\begin{tabular}[t]{l}\textcolor{red}{$s_e$}\end{tabular}}}}%
    \put(0.42988303,0.18998899){\color[rgb]{0,0,0}\makebox(0,0)[lt]{\lineheight{1.25}\smash{\begin{tabular}[t]{l}\textcolor{blue}{$m_e$}\end{tabular}}}}%
    \put(0.22332075,0.49983632){\color[rgb]{0,0,0}\makebox(0,0)[lt]{\lineheight{1.25}\smash{\begin{tabular}[t]{l}\textcolor{blue}{$m_a$}\end{tabular}}}}%
    \put(0.22332075,0.07217081){\color[rgb]{0,0,0}\makebox(0,0)[lt]{\lineheight{1.25}\smash{\begin{tabular}[t]{l}\textcolor{blue}{$m_b$}\end{tabular}}}}%
    \put(0.72988965,0.06559118){\color[rgb]{0,0,0}\makebox(0,0)[lt]{\lineheight{1.25}\smash{\begin{tabular}[t]{l}\textcolor{blue}{$m_c$}\end{tabular}}}}%
    \put(0.72988965,0.46226519){\color[rgb]{0,0,0}\makebox(0,0)[lt]{\lineheight{1.25}\smash{\begin{tabular}[t]{l}\textcolor{blue}{$m_d$}\end{tabular}}}}%
    \put(0,0){\includegraphics[width=\unitlength,page=2]{FoliInPair.pdf}}%
  \end{picture}%
\endgroup%

 \caption{The \mf\ in a pair of adjacent ideal triangles.}
 \label{fig:mf-in-pair}
\end{figure}

A direct corollary is:

\begin{thm}\label{prop:wp-in-sh}
The \WP volume form under the \Sh\ is equal to the Euclidean volume form on $\mathcal{C}_\Delta$, up to a scaling constant.
\end{thm}

\subsection{Proof of Theorem \ref{thm:main} }
\indent

\begin{proof}

It is obvious that the shear norm $\|Sh_\Delta\|$ is proper in $\mathcal{T}_{g,n}$. Applying Theorem \ref{thm:shear-apl} and Proposition \ref{prop:bound} to Theorem \ref{thm:count},
we have
\begin{equation*}
\lim_{L\to + \infty}
\frac{ \#\left\{\ \phi\in \mathrm{Mod}_{g,n}\ \big|\ \|Sh_\Delta(\phi\cdot X)\|\leqslant L\ \right\} }
{L^{6g-6+2n}}
=
\frac{n_\Delta \cdot B(X)}{b_{g,n}},
\end{equation*}
where
\begin{equation*}\label{eq:coeff}
n_\Delta=
\lim_{L\to + \infty}
\frac{ \mu_{wp} \left\{ Y\in\mathcal{T}_{g,n}\ \big|\ \|Sh_\Delta(Y)\|\leqslant L \right\} }
{L^{6g-6+2n}}\ .
\end{equation*}

We have shown in \Thm\ \ref{prop:wp-in-sh} that, up to a scaling constant,
the Weil-Petersson volume form is equal to the Euclidean volume form on the image of the \Sh, which is obviously homogenous.
Thus the coefficient in (\ref{eq:coeff}) is equal to the volume of the unit ball.
This finishes the proof.

\end{proof}

\phantomsection
\addcontentsline{toc}{section}{Refence}

\end{document}